\documentclass{amsart}
\usepackage{amsmath}
\usepackage{amssymb}
\usepackage{amsthm}
\usepackage{bbm}
\usepackage{cite}
\usepackage{leftidx}
\usepackage{pgf}

\usepackage{tikz-cd}
\usepackage{verbatim}

\usepackage{hyperref}

\author{Clark Barwick and Saul Glasman}

\title{A note on stable recollements}

\newcommand{\D}{\Delta}

\DeclareMathOperator{\Fun}{Fun}

\newcommand{\id}{\text{id}}

\newcommand{\inc}{\subseteq}

\newcommand{\iy}{\infty}

\newcommand{\mb}{\mathbf}
\newcommand{\mc}{\mathcal}
\newcommand{\mf}{\mathfrak}
\newcommand{\Map}{\text{Map}}

\newcommand{\os}{\overset}

\newcommand{\X}{\times}

\theoremstyle{definition}

\newtheorem{rem}[section]{Remark}

\newtheorem{war}[section]{Warning}

\theoremstyle{plain}

\newtheorem{cor}[section]{Corollary}
\newtheorem{lem}[section]{Lemma}
\newtheorem{prop}[section]{Proposition}

\begin{document}

\begin{abstract}
In this short \'etude, we observe that the full structure of a recollement on a stable $\iy$-category can be reconstructed from minimal data: that of a reflective and coreflective full subcategory. The situation has more symmetry than one would expect at a glance. We end with a practical lemma on gluing equivalences along a recollement.
\end{abstract}

\maketitle

Let $\mb{X}$ be a stable $\iy$-category and let $\mb{U}$ be a full subcategory of $\mb{X}$ that is stable under equivalences and is both reflective and coreflective -- that is, its inclusion admits both a left and a right adjoint. We'll denote the inclusion functor $\mb{U}\subseteq\mb{X}$ by $j_{\ast}$ and its two adjoints by $j^{\ast}$ and $j^{\times}$, so that we have a chain of adjunctions
\[ j^{\ast} \dashv j_{\ast} \dashv j^{\times}. \]

 Let $\mb{Z}^{\wedge}\subseteq\mb{X}$ denote the right orthogonal complement of $\mb{U}$ -- that is, the full subcategory of $\mb{X}$ spanned by those objects $M$ such that $\Map_{\mb{X}}(N, M) = \ast$ for every $N \in \mb{U}$. Dually, let $\mb{Z}^{\vee}\subseteq\mb{X}$ denote the left orthogonal complement of $\mb{U}$ --  that is, the full subcategory of $\mb{X}$ spanned by those objects $M$ such that $\Map_{\mb{X}}(M, N) = \ast$ for every $N \in \mb{U}$. The inclusions of $\mb{Z}^{\wedge}\subseteq\mb{X}$ and $\mb{Z}^{\vee}\subseteq\mb{X}$ will be denoted $i_{\wedge}$ and $i_{\vee}$ respectively.

\begin{war} Our notation is chosen to evoke a geometric idea, but the role of open and closed is reversed from recollements that arise in the theory of constructible sheaves.

In our thinking, we imagine $\mb{X}$ as the $\infty$-category $\mb{D}_{\textit{qcoh}}(X)$ of quasicoherent complexes over a suitably nice scheme $X$, which is decomposed as an open subscheme $U$ together with a closed complement $Z$. In this analogy, we think of $\mb{U}$ as the $\infty$-category of quasicoherent modules on $U$, embedded via the (derived) pushforward. The subcategory $\mb{Z}^{\vee}$ is then the $\infty$-category of quasicoherent complexes on $X$ that are set-theoretically supported on $Z$, and the subcategory $\mb{Z}^{\wedge}$ is the $\infty$-category of quasicoherent complexes on $X$ that are complete along $Z$.
\end{war}

\begin{lem}\label{lem:adjs}
	In this situation, $\mb{Z}^{\wedge}$ is reflective and $\mb{Z}^{\vee}$ is coreflective.
\end{lem}
\begin{proof}
	Denote by $\kappa$ the cofiber of the counit $j_{\ast}j^{\times} \to \id_{\mb{X}}$. Then $\kappa(\mb{X}) \inc \mb{Z}^{\wedge}$, so we factor
	\[ \kappa = i_{\wedge} i^{\wedge} \]
	with $i^{\wedge}\in\Fun(\mb{X},\mb{Z}^{\wedge})$. We claim that $i^{\wedge}$ is left adjoint to $i_{\wedge}$. Indeed, for any $M \in \mb{X}$ and $N \in \mb{Z}^{\wedge}$, we have a cofiber sequence of spectra
	\[ F_{\mb{Z}^{\wedge}}(i^{\wedge} M, N) \simeq F_{\mb{X}}(i_{\wedge} i^{\wedge} M, i_{\wedge} N) \to F_{\mb{X}}(M, i_{\wedge} N) \to F_{\mb{X}}(j_{\ast}j^{\times}M, i_{\wedge} N )\simeq 0. \]

	The proof that $\mb{Z}^{\vee}$ is coreflective is dual, and we'll denote the right adjoint of $i_{\vee}$ by $i^{\vee}$.
\end{proof}
\begin{lem}
	In the sense of \cite[Df. 3.4]{Gla15a}, 
	\[
	\mf{S}(\{0\}) = \mb{Z}^{\wedge},\ \mf{S}(\{1\}) = \mb{U},\ \mf{S}(\D^{1}) = \mb{X},\ \mf{S}(\emptyset) = 0
	\]
	is a stratification of $\mb{X}$ along $\D^1$.
\end{lem}
\begin{proof}
	After unravelling the notation, one sees that this amounts to the following two claims.
	\begin{itemize}
		\item First, $i^{\wedge} j_{\ast} j^{\ast} = 0$. This point is obvious.
		\item The usual fracture square
		\[\begin{tikzcd} \id \ar{r} \ar{d} & i_{\wedge} i^{\wedge} \ar{d} \\
		j_{\ast} j^{\ast} \ar{r} & j_{\ast} j^{\ast} i_{\wedge} i^{\wedge} \end{tikzcd} \]
		is cartesian. To see this, take fibers of the horizontal maps to get the map
		\[ j_{\ast} j^{\times} \to j_{\ast} j^{\ast} j_{\ast} j^{\times}, \]
		which is an equivalence since $j^{\ast} j_{\ast}$ is homotopic to the identity.\qedhere
	\end{itemize}
\end{proof}
\begin{rem}
	Conversely, if $\mf{S}$ is a stratification of $\mb{X}$ along $\D^1$, then $\mf{S}(\{0\})$ is coreflective as well as reflective. Indeed, the fracture square together with the argument of Lm. \ref{lem:adjs} shows that the fiber of $\id \to \mc{L}_1$ defines a right adjoint to the inclusion of $\mf{S}(\{0\})$.
\end{rem}
\begin{lem}
	In the sense of \cite[Df. A.8.1]{HA}, $\mb{X}$ is a recollement of $\mb{U}$ and $\mb{Z}^{\wedge}$.
\end{lem}
\begin{proof}
	The only claim that isn't obvious is point e): that $j^{\ast}$ and $i^{\wedge}$ are jointly conservative. But since they are exact functors of stable $\iy$-categories, this is equivalent to the claim that if $j^{\ast}M$ and $i^{\wedge}M$ are both zero, then $M$ is zero, and this is clear from the fracture square.
\end{proof}
\begin{rem}
	Again there's a converse; indeed, if a stable $\infty$-category $\mb{X}$ is a recollement of $\mb{U}$ and $\mb{Z}$, then $\mb{U}$ is coreflective \cite[Rk. A.8.5]{HA}. We thus conclude that the following three pieces of data are essentially equivalent:
	\begin{itemize}
	\item reflective and coreflective subcategories of $\mb{X}$,
	\item stratifications $\mf{S}$ along $\Delta^1$ in the sense of \cite[Df. 3.4]{Gla15a} with $\mf{S}(\Delta^1)=\mb{X}$, and
	\item recollements of $\mb{X}$ in the sense of \cite[Df. A.8.1]{HA}.
	\end{itemize}
\end{rem}

As we have described this structure, there's a surprising intrinsic symmetry that traditional depictions of recollements don't really bring out:
\begin{prop}
	The functors $i^{\wedge} i_{\vee}$ and $i^{\vee} i_{\wedge}$ define inverse equivalences of categories between $\mb{Z}^{\wedge}$ and $\mb{Z}^{\vee}$.
\end{prop}
This proposition is an extreme abstraction of prior results, such as those of \cite{DG02}, giving equivalences between categories of complete objects and categories of torsion objects.
\begin{proof}
	Let's show that the counit map 
	\[ \eta\colon i^{\wedge}i_{\vee}i^{\vee}i_{\wedge} \to \id \]
	is an equivalence; the other side will of course be dual. The counit factors as
	\[ i^{\wedge}i_{\vee}i^{\vee}i_{\wedge} \os{\eta_0}\longrightarrow i^{\wedge}i_{\wedge} \os{\eta_1}\longrightarrow \id, \]
	but of course $\eta_1$ is an equivalence since $i_{\wedge}$ is fully faithful. But $\eta_0$ fits into a cofiber sequence
	\[ i^{\wedge}i_{\vee}i^{\vee}i_{\wedge} \os{\eta_0}\longrightarrow i^{\wedge}i_{\wedge} \to i^{\wedge}j_{\ast}j^{\ast}i_{\wedge}, \]
	and the final term is zero since $i^{\wedge} j_{\ast} = 0$.
\end{proof}

Finally, we give a useful criterion for when a morphism of recollements gives rise to an equivalence, the proof of which is unfortunately a little more technical than the foregoing.
\begin{prop} \label{prop:recequ}
	Let $\mb{X}$ and $\mb{X}'$ be stable $\iy$-categories with reflective, coreflective subcategories $\mb{U}\subseteq\mb{X}$ and $\mb{U}'\subseteq\mb{X}'$ and ancillary subcategories
	\[\mb{Z}^{\vee}\subseteq\mb{X},\ \mb{Z}^{\wedge}\subseteq\mb{X},\ (\mb{Z}')^{\vee}\subseteq\mb{X}',\ (\mb{Z}')^{\wedge}\subseteq\mb{X}'.\]
	Suppose $F \colon\mb{X} \to \mb{Y}$ is a functor with
	\[F(\mb{U}) \inc \mb{U}',\ F(\mb{Z}^{\wedge}) \inc (\mb{Z}')^{\wedge},\ F(\mb{Z}^{\vee}) \inc (\mb{Z}')^{\vee}. \]
	Suppose moreover that $F|_{\mb{U}}$ and at least one of $F|_{\mb{Z}^{\wedge}}$ and $F|_{\mb{Z}^{\vee}}$ is an equivalence. Then $F$ is an equivalence.
\end{prop}
\begin{proof}
	Let's suppose that $F|_{\mb{Z}^{\wedge}}$ is an equivalence; once again, the other case is dual. 
	\begin{lem} Set 
	\[ \mb{Z}^{\wedge}\downarrow_{\mb{X}}\mb{U} = \mb{Z}^{\wedge} \X_{\mb{X}} \Fun(\D^1, \mb{X}) \X_{\mb{X}} \mb{U}\]
	be the $\infty$-category of morphisms in $\mb{X}$ whose source is in $\mb{Z}^{\wedge}$ and whose target is in $\mb{U}$; we claim that the functor
	\[ k \colon \mb{Z}^{\wedge}\downarrow_{\mb{X}}\mb{U} \to \mb{X} \]
	that maps a morphism to its cofiber is an equivalence. \end{lem} 
	\begin{proof} The functor $k$ is really constructed as a zigzag
		\[ \mb{Z}^{\wedge}\downarrow_{\mb{X}}\mb{U} \os{\sim} \longleftarrow \mb{E} \os{t}\longrightarrow \mb{X}, \]
		where $\mb{E}$ is the $\infty$-category of cofiber sequences $M \to N \to P$ in $\mb{X}$ for which $(M \to N) \in \mb{Z}^{\wedge}\downarrow_{\mb{X}}\mb{U}$. The leftward arrow is a trivial Kan fibration. We'd like to prove that the right hand arrow, $t$, is also a trivial Kan fibration. It's clearly a cartesian fibration, and so it suffices to show that each fiber of $t$ is a contractible Kan complex.
		
		The fiber of $t$ over $P$ is the $\infty$-category of cofiber sequences 
		\[ M \to N \to P \]
		with $M \in \mb{Z}^{\wedge}$ and $N \in \mb{U}$. Since fibers are unique, this is equivalent to the $\infty$-category of morphisms $\phi \colon N \to P$ with $N \in \mb{U}$ and $\text{fib}(\phi) \in \mb{Z}^{\wedge}$. But $\text{fib}(\phi) \in \mb{Z}^{\wedge}$ if and only if $\phi$ exhibits $N$ as the $\mb{U}$-colocalization of $P$, and such a $\phi$ exists uniquely.
	\end{proof} 
	\begin{cor}
		The $\infty$-category $\mb{X}$ is equivalent to the $\infty$-category of sections of the map
		\[p\colon \mb{C} \to \D^1 \]
		where $\mb{C}\subseteq\mb{X} \X \D^1$ is the full subcategory spanned by objects of $\mb{Z}^{\wedge} \X\{0\}$ or $\mb{U} \X \{1\}.$ \qedhere
	\end{cor}
	Observe here that $p$ is a cocartesian fibration, and the cocartesian edges correspond to morphisms $f\colon M \to N$ in $\mb{X}$ which exhibit $N$ as the $\mb{U}$ localization of $M$.
	
	Now we finish the proof of Pr. \ref{prop:recequ}. In fact, $ F\colon \mb{X} \to \mb{X}'$ induces a functor over $\D^1$
	\[\overline{F} \colon \mb{C} \to \mb{C}', \]
	where $\mb{C}'\subseteq\mb{X}' \X \D^1$ is the full subcategory spanned by objects of $(\mb{Z}')^{\wedge} \X\{0\}$ or $\mb{U}' \X \{1\}$. By hypothesis, $\overline{F}$ induces equivalences on the fibers over $\{0\}$ and $\{1\}$. If $\overline{F}$ moreover preserves cocartesian edges, we'll be able to conclude that $\overline{F}$ is an equivalence of $\infty$-categories, inducing an equivalence on $\infty$-categories of sections, whence the result.
	
	The claim that $\overline{F}$ preserves cocartesian edges is equivalent to the claim that the naturally lax-commutative square
	\[ \begin{tikzcd} \mb{Z}^{\wedge} \ar{r}{j^{\ast}i_{\wedge}} \ar{d}[left]{F|_{\mb{Z}^{\wedge}}} & \mb{U} \ar{d}{F|_{\mb{U}}} \\
	(\mb{Z}')^{\wedge} \ar{r}[below]{(j')^{\ast}(i')^{\wedge}} & \mb{U}' \end{tikzcd} \]
	is in fact commutative up to equivalence. In fact, the stronger claim that the lax-commutative square
	\[ \begin{tikzcd} \mb{X} \ar{r}{j^{\ast}} \ar{d}[left]{F} & \mb{U} \ar{d}{F|_{\mb{U}}} \\
	\mb{X}' \ar{r}[below]{(j')^{\ast}} & \mb{U}' \end{tikzcd} \]
	commutes up to equivalence is equivalent to the claim that $F$ takes $j^{\ast}$-equivalences to $(j')^{\ast}$-equivalences. But this is the case if and only if $F$ takes left orthogonal objects to $\mb{U}$ -- that is, objects of $\mb{Z}^{\vee}$ -- to left orthogonal objects to $\mb{U}'$ -- that is, objects of $(\mb{Z}')^{\vee}$. Since this was one of our hypotheses, the proof is complete.\qedhere
\end{proof}

\bibliographystyle{plain}
\bibliography{mybib}
\end{document}